\documentclass{amsart}
\usepackage{amsthm,amssymb,verbatim}
\theoremstyle{plain}\newtheorem{Theorem}{Theorem}[section]
\theoremstyle{plain}
\theoremstyle{plain}\newtheorem{Corollary}[Theorem]{Corollary}
\theoremstyle{plain}\newtheorem{Lemma}[Theorem]{Lemma}
\theoremstyle{plain}\newtheorem{Proposition}[Theorem]{Proposition}
\theoremstyle{definition}
\theoremstyle{definition}
\theoremstyle{definition}
\theoremstyle{definition}\newtheorem{Remark}[Theorem]{Remark}
\theoremstyle{definition}
\theoremstyle{definition}
\theoremstyle{definition}
\theoremstyle{definition}
\theoremstyle{definition}
\theoremstyle{definition}
\theoremstyle{definition}
\theoremstyle{definition}
\theoremstyle{definition}\newtheorem{Notation/Definition}
[Theorem]{Notation/Definition}
\theoremstyle{definition}\newtheorem{noth}[Theorem]{}

\def\CH{{\mathcal{H}}}

\def\F{{\mathbb F}}

\def\Z{{\mathbb Z}}

\def\dim{\mathrm{dim}}

\def\Hom{\mathrm{Hom}}           
\def\ker{\mathrm{ker}}

\def\rad{\mathrm{rad}}

\def\soc{\mathrm{soc}}

\newcommand{\PSL}{\operatorname{PSL}}
\newcommand{\SL}{\operatorname{SL}}

\newcommand{\PSU}{\operatorname{PSU}}
\newcommand{\SU}{\operatorname{SU}}

\newcommand{\Sp}{\operatorname{Sp}}

\def\bfG{\mathbf{G}}
\def\bfB{\mathbf{B}}
\def\bfU{\mathbf{U}}
\def\bfT{\mathbf{T}}
\def\LL{\text{\rm LL}}

\begin{document}
\title{The projective cover of the trivial representation
for a finite group of \\ Lie type in defining characteristic
}
\date{\today}
\author{{Shigeo Koshitani and J\"urgen M\"uller}}
\address{Department of Mathematics, Graduate School of Science,
Chiba University, 1-33 Yayoi-cho, Inage-ku, Chiba 263-8522, Japan.}
\email{koshitan@math.s.chiba-u.ac.jp}
\address{Arbeitsgruppe Algebra und Zahlentheorie,
Bergische Universit\"at Wuppertal,
Gau{\ss}-Stra{\ss}e 20, D-42119 Wuppertal, Germany.}
\email{juergen.mueller@math.uni-wuppertal.de}
\thanks{}
\keywords{Modular representations, projective cover of the trivial module,
finite groups of Lie type in defining characteristic}
\subjclass[2000]{20C05,20C20,20C33}

%%%%%%%%%%%%%%%%%%%%%%%%%%%%%%%%%%%%%%%%%%%%%%%%%%%%%%%%%%%%%%%%%%%%%%%%%%%%%
\begin{abstract} 
We give a lower bound of the Loewy length of the projective cover of
the trivial module for the group algebra $kG$ of a finite group
$G$ of Lie type defined over a finite field of odd characteristic $p$,
where $k$ is an arbitrary field of characteristic $p$.
The proof uses Auslander-Reiten theory.
%We point out that in one of the main results of a recent paper by
%L{\"u}beck--Malle [2016] there is a gap, implying that contrary to the
%claim made there a question asked in
%Koshitani--K{\"u}lshammer--Sambale [2014] is still open.
%and Lassueur-Malle (2015).
\end{abstract}

\maketitle

%%%%%%%%%%%%%%%%%%%%%%%%%%%%%%%%%%%%%%%%%%%%%%%%%%%%%%%%%%%%%%%%%%%%%%%%%%%%%
\section{Introduction}

One of the main problems in the modular representation theory
of finite groups is obtaining classes of finite groups $G$
such that the group algebras $kG$ have specific ring-theoretical 
properties, where $k$ is a field of positive characteristic $p$,
just as Brauer stated in \cite[Problem 16]{Brauer1963}.
Examples of such properties are the Loewy length $\LL(P(k_G))$ 
of the projective cover $P(k_G)$ of the trivial $kG$-module $k_G$,
and the `first' Cartan invariant $c_{11}(G):=[P(k_G)\colon k_G]$ 
of the group algebra $kG$, that is the multiplicity of 
$k_G$ as a composition factor of $P(k_G)$.

In the present paper, we are interested in the questions for
representations of finite groups of Lie type in their defining
characteristic;
see also for example \cite[(2.6)]{Humphreys1985} and 
\cite[Section 11.4]{Humphreys2006}. In even characteristic we 
are in a good shape indeed: Let $G$ be a simple finite group of 
Lie type defined over a finite field of characteristic $p=2$.
Then it follows from results of Okuyama \cite{Okuyama} and
Erdmann \cite{Erdmann1977} that $\LL(P(k_G))=3$ if and only if $G=\SL_3(2)$,
see \cite[Theorem 3.3]{Koshitani2014}; and by 
\cite[Theorem 1.2]{Koshitani2014} we always have $\LL(P(k_G))\neq 4$.
Moreover, by \cite[Lemma 4.5]{Koshitani2015} we have 
$c_{11}(G)=2$ if and only if $G=\SL_3(2)$.
Hence in conclusion for this class of groups we have 
$\LL(P(k_G))\geq 5$ if and only if $c_{11}(G)\geq 3$ 
if and only if $G{\not\cong}\,\SL_3(2)$.

Here we are now interested in the odd characteristic case,
and the purpose of this paper is to present the following theorem.

\begin{Theorem}\label{Lie}
Assume that $p$ is an odd prime, and $G$ is a simple or an almost 
simple finite group of Lie type (in the sense of \ref{titsthm})
defined over a finite field of characteristic $p$, such that 
the Sylow $p$-subgroups of $G$ are non-cyclic, that is 
$G\not\in\{\SL_2(p),{}^2 G_2(\sqrt{3})'\}$.
Then the projective cover $P(k_G)$ of the trivial $kG$-module
has Loewy length $\LL(P(k_G))\geq 5$.
\end{Theorem}

%As it turns out, 
This sheds new light on a couple of our earlier 
results in \cite{KoshitaniKuelshammerSambale2014}. Firstly,
we get the following, providing an alternative proof of 
\cite[Proposition 4.12]{KoshitaniKuelshammerSambale2014}:

\begin{Corollary}
%{\bf (\cite[Proposition 4.12]{KoshitaniKuelshammerSambale2014})}
\label{4_12}
Under the same assumptions as in Theorem \ref{Lie}, the principal
block algebra $B_0(G)$ of $kG$ has the Loewy length $\LL(B_0(G))\geq 5$.
\end{Corollary}

Secondly, our method of proof of Theorem \ref{Lie} also allows to generalize
\cite[Proposition 4.10]{KoshitaniKuelshammerSambale2014}
(which was proved under the additional assumption $p\geq 5$):

\begin{Theorem}\label{4_10}
Let $p$ be an odd prime, let $G$ be a finite group such that $O_{p'}(G)=\{1\}$, 
and assume that the principal block algebra $B_0(G)$ of $kG$ has 
Loewy length $\LL(B_0(G))=4$. Then $F^*(G)=E(G)=O^{p'}(G)$
is a non-abelian simple group.
\end{Theorem}

Our line of reasoning in proving the above results can be summarized
as follows: Letting $G$ be a finite group, 
assume that $\LL(B_0(G))\leq 4$. Then we have $\LL(P(k_G))\leq 4$ as well,
from which we conclude that the heart %by Proposition \ref{llprop} 
$$ \CH(k_G):=\rad(P(k_G))/\soc(P(k_G)) $$ 
of $P(k_G)$
actually is a simple $kG$-module $S$, where apart from a trivial
exceptional case we have $S\,{\not\cong}\,k_G$.
%In order to prove Theorem \ref{Lie} we then apply the 
If now $G$ is a finite group of Lie type, we are in a position to apply the 
Kawata--Michler--Uno Theorem \cite{KawataMichlerUno2001},
% which we restate as Theorem \ref{KawataMichlerUno}.
which in turn is a specific application of Auslander-Reiten theory
to this class of groups.
Actually, to cover all the cases allowed by our assumptions
we have to extend the Kawata--Michler--Uno Theorem slightly,
which is done by unraveling the strategy of proof in 
\cite{KawataMichlerUno2001}, and checking the necessary conditions
for the additional cases.

This paper thus is organized as follows:
In Section \ref{lietypegrp} we collect the necessary prerequisites 
from the theory of finite groups of Lie type and their representations
in defining characteristic. Moreover, we revisit the 
Kawata--Michler--Uno heorem and its proof, in order to extend it to
the class of groups considered here.
In Section \ref{proofs} we then first prove the general statement on 
$\CH(k_G)$ mentioned in the previous paragraph, and subsequently 
give proofs of our main results.
%In Section \ref{concl} we give a few concluding remarks.

We would like to stress the fact that our approach stays entirely 
in the realm of modular representation theory, no results of 
ordinary representation theory or character theory are needed, not 
even behind the scenes to prove the Kawata--Michler--Uno Theorem.
Moreover, apart from the case-by-case analysis needed to prove the latter,
our approach is completely structural, without case distinctions.

%while to prove Theorem \ref{4_10} we make appeal to the
%Alperin--Collins--Sibley Theorem \cite{ACS},
%which holds for arbitrary finite groups.

\begin{noth}{\bf Cartan invariants.}\label{example}
%Keep the notation in Theorem \ref{Lie} and Corollary \ref{4_12}.
In view of the comments on the even characteristic case 
at the beginning of this section, we have a quick look onto
Cartan invariants:

(a)
Let $p$ be odd and $G$ a finite group.
Then, apart from the above-mentioned exception,
we conclude that in general the condition $\LL(P(k_G))\leq 4$
implies that the Cartan invariant $c_{11}(G)=2$. Hence the question arises
whether the condition $\LL(P(k_G))\geq 5$ implies $c_{11}(G)\geq 3$.
But this is {\it not} true in general, in other words $c_{11}(G)=2$ 
is a strictly weaker condition than $\LL(P(k_G))\leq 4$:
%at least as far as the class of all finite groups is concerned:

One of the smallest counter-examples is $G:=C_3^2\colon Q_8$ for $p=3$,
where $C_3$ and $Q_8$ are the cyclic group of order $3$ and the
quaternion group of order $8$, respectively, 
and the action of $Q_8$ on the elementary abelian group $C_3^2$ 
is regular. Then we have $c_{11}(G)=2$, although $\LL(P(k_G))=5$;
actually, all projective indecomposable $kG$-modules in $B_0(G)$
have Loewy length $5$, so that $\LL(B_0(G))=5$ as well;
see also \cite[Lemma (4.2) and Lemma (4.3)]{KoshitaniKunugi2001}, 
where infinitely many such examples are given.

%Of course, the picture might change, if we restrict ourselves to
%the class of quasi-simple finite groups of Lie type in odd characteristic.
%But so far we are not able to prove this. We come back to this
%in Remark \ref{mallegap}, where we point out a gap related to this
%in the recent paper \cite{LuebeckMalle2016} by L\"ubeck--Malle,
%implying that a question raised in \cite{KoshitaniKuelshammerSambale2014} 
%is still open, contrary to the claim made in \cite{LuebeckMalle2016}.
%\cite{LassueurMalle2015}. 

(b)
The present paper arose out of an attempt
to approach the question raised in our earlier paper
\cite{KoshitaniKuelshammerSambale2014},  
whether or not for any non-abelian finite simple group $G$ 
having non-cyclic Sylow $p$-subgroups, where $p$ is odd, 
the Cartan invariant $c_{11}(G)\geq 3$.
Unfortunately, by the above comments, the condition $c_{11}(G)=2$ 
is strictly weaker than $\LL(P(k_G))\leq 4$,
so that after all Theorem \ref{Lie} does not help here.

Moreover, this is also related to the recent result 
\cite[Theorem 7.1]{LuebeckMalle2016} %by L\"ubeck--Malle,
saying that, if $G$ is a classical simple finite group of Lie type, 
and $l$ is an odd prime such that the Sylow $l$-subgroups are non-cyclic,
then the $l$-modular Cartan invariant $c^{(l)}_{11}(G)\geq 3$.
But in the proof given there the defining characteristic case $l=p$ is
erroneously attributed to
\cite[Proposition 4.12]{KoshitaniKuelshammerSambale2014},
which only states that the Loewy length of the block algebra $B_0(G)$ 
as a whole is $\LL(B_0(G))\,{\not=}\,4$, but nothing comprehensive 
is said about $P(k_G)$ or $c_{11}(G)$.  
Actually, when we were preparing the present manuscript, we 
noticed the above-mentioned gap, and informed the authors of 
\cite{LuebeckMalle2016} about it, and after this manuscript was 
completed we learned that they have been able to close the gap in
\cite{LuebeckMalle2016corr}.
We would like to point out that \cite{LuebeckMalle2016,LuebeckMalle2016corr}
make heavy use of ordinary character theory of finite groups of Lie
type, so that their approach is fundamentally different from ours.
\end{noth}

\begin{noth}{\bf Notation and terminology.}
We shall in particular use the following notation. 
Let $k$ be an arbitrary %algebraically closed 
field of positive characteristic $p$, and 
let $A$ be a finite-dimensional $k$-algebra.
Unless stated otherwise we mean by an $A$-module  
a finitely generated right $A$-module. 
%We write ${\mathrm{rad}}(A)$ for the Jacobson radical of $A$.
If $M$ is an $A$-module,
% we write $N \mid M$, 
% if $N$ is isomorphic to a direct summand of $M$. % as $A$-modules.
we write $\rad(M)$ for the Jacobson radical 
of $M$, and $P(M)$ for the projective cover of $M$. 
We say that $n=\LL(M)$ is the Loewy length of $M$,
if $n\geq 0$ is smallest with $\rad^n(M)=\{0\}$.

Moreover, if $G$ is a finite group, 
we write $Z(G)$ for the center of $G$, and $F(G)$ for the 
Fitting subgroup of $G$. 
%We denote by $H\tchar G$ if $H$ is a characteristic subgroup of $G$,
%and by $H\subn G$ if $H$ is a subnormal subgroup of $G$.
We let $E(G)$ be the layer of $G$, that is
the central product of the components of $G$, where a component of $G$
is a subnormal quasi-simple subgroup, and
we write $F^*(G)=F(G)E(G)$ for the generalized Fitting subgroup of $G$.
%see \cite[Chapter 6, Definition 6.8 and (6.9)]{Suzuki1982}.
We denote by $k_G$ the trivial $kG$-module,
and by $B_0(G):=B_0(kG)$ the principal block algebra of $kG$.
% and we call the projective cover $P(k_G)$ the $1$-PIM for $kG$.
Given a $kG$-module $M$, we write $M^*:={\mathrm{Hom}}_k(M,k)$ 
for the dual module of $M$, which becomes a {\it right} $kG$-module
with resepct to the contragredient acion, 
and $M$ is called self-dual if $M\cong M^*$ as $kG$-modules.

For other general notation and terminology we refer \cite{NagaoTsushima} 
and \cite{Suzuki1982}, as far as representation theory and finite
group theory are concerned, respectively. Moreover, for the necessary
background on finite groups of Lie type and Auslander-Reiten theory,
we refer to \cite{MalleTesterman} and \cite{ARS}, respectively.
\end{noth}

%%%%%%%%%%%%%%%%%%%%%%%%%%%%%%%%%%%%%%%%%%%%%%%%%%%%%%%%%%%%%%%%%%%%%%%%%%%%
\section{Groups of Lie type in defining characteristic}\label{lietypegrp}

We collect the facts needed from the theory of finite groups of Lie type
and their representations.

\begin{noth}{\bf Tits's Theorem.}\label{titsthm}
Let $\bfG$ be a simply-connected simple linear algebraic group
over the algebraic closure $\overline{\F}_p$ of the field $\F_p$
with $p$ elements. Let $F\colon\bfG\rightarrow\bfG$ be a
Steinberg endomorphism, see \cite[Ch.21]{MalleTesterman}.
Let $\bfT$ be an $F$-stable maximal torus of $\bfG$, contained
in an $F$-stable Borel subgroup $\bfB$ of $\bfG$, and let 
$\bfU$ be the unipotent radical of $\bfB$.
Let $q$ be the absolute value of the eigenvalues of $F$
for its action on the character group of $\bfT$. 
Then $q$ is an integral power of $p$, except in
the `very-twisted cases' giving rise to the Suzuki and Ree groups,
where it is an half-integral power. 
The associated set of fixed points $\bfG(q):=\bfG^F$ is
called a {\it finite group of Lie type}.

Now Tits's Theorem, see \cite[Theorem 24.17]{MalleTesterman}, 
says that except in the cases
$$ \SL_2(2),\, \SL_2(3),\, \SU_3(2),\, \Sp_4(2),\, G_2(2),\,
{}^2 B_2(\sqrt{2}),\, {}^2 G_2(\sqrt{3}),\, {}^2 F_4(\sqrt{2}) $$
the group $G:=\bfG(q)$ is perfect, implying that $G$ is quasi-simple,
that is $G/Z(G)$ is non-abelian simple; note that $Z(G)$ always is a
$p'$-group. In this case, $G$ is called a {\it quasi-simple}
finite group of Lie type.

The non-solvable groups amongst the above exceptions, that is
$$ \Sp_4(2),\, G_2(2),\, {}^2 G_2(\sqrt{3}),\, {}^2 F_4(\sqrt{2}) ,$$
all turn out to have trivial center, and their derived subgroups
have index $p$ and are non-abelian simple. These groups are
called the {\it almost simple} finite groups of Lie type.
In both the quasi-simple and the almost simple cases the group
$(G/Z(G))'$ is called a {\it simple} group of Lie type; note that 
this encompasses the Tits group ${}^2 F_4(\sqrt{2})'$, which does 
not occur elsewhere in the classification of finite simple groups.
\end{noth}

\begin{Proposition}{\bf Finite and tame cases; 
see also \cite[Section 8.9]{Humphreys2006}.}\label{finitetame}
Let $G$ be a simple or an almost simple finite group of Lie type, 
and let $U\leq G$ be a Sylow $p$-subgroup of $G$. Then the following holds:

(a) If $p$ is arbitrary, then $U$ is cyclic if and only if 
one of the following holds:

$\circ$
$G\cong\SL_2(p)$, in which case $U\cong C_p$ has order $p$;

$\circ$
$G\cong {}^2 G_2(\sqrt{3})'$,
in which case $U\cong C_9$ has order $9$.

(b) If $p=2$, then $U$ is dihedral, semi-dihedral or generalized quaternion 
if and only if one of the following holds:

$\circ$
$G\cong\SL_2(4)$, 
in which case we have $U\cong C_2^2$, the Klein $4$-group;

$\circ$
$G\cong\SL_3(2)$, 
in which case we have $U\cong D_8$, the dihedral group of order $8$;

$\circ$
$G\cong\Sp_4(2)'$,
in which case we have $U\cong D_8$, the dihedral group of order $8$.
\end{Proposition}

\begin{proof}[\bf Proof]
Although there is a thorough discussion of the structure of $U$ in 
\cite[Section 3.3]{GorLyoSol}, we choose a direct approach, 
tailored for our purposes. The almost simple groups and their
derived subgroups being discussed in \ref{lost} below,
we may assume that $G=\bfG^F$ is quasi-simple, and
that $U:=\bfU^F$, by \cite[Corollary 24.11]{MalleTesterman}.

(a) 
Assume that $U$ is cyclic.
Then it follows from \cite[Proposition 23.7, Corollary 23.9]{MalleTesterman}
that $G$ has twisted Lie rank $1$, thus $U$ is a root subgroup.
For the structure of the root subgroups occurring we refer to 
\cite[Example 23.10]{MalleTesterman} and 
\cite[Proposition 13.6.3, Proposition 13.6.4]{Carter}.
We now consider the various types:

For type $A_1$ we have $G\cong\SL_2(q)$ and $U\cong\F_q^+$,
%is elementary abelian of order $q$,
hence $U$ is cyclic if and only if $q=p$.
For type ${}^2 A_2$ we have $G\cong\SU_3(q)$ and 
$U/U'\cong\F_{q^2}^+$, hence $U$ is not cyclic.
For types ${}^2 B_2$ and ${}^2 G_2$ we have $q=\sqrt{2}^{2f+1}$
and $q=\sqrt{3}^{2f+1}$, respectively, for some $f\geq 1$, and 
$U/U'\cong\F_{q^2}^+$, hence $U$ is not cyclic either.
%but already ${}^2 G_2(\sqrt{3})\cong\PSL_2(8):3$ 
%has Sylow $3$-subgroups which are extra-special of order $27$.

(b) 
Let $p=2$, and assume that $U$ is
dihedral, semi-dihedral or generalized quaternion.
Then, by \cite[Satz I.14.9(b)]{Huppert}, we have $U/U'\cong C_p^2$. 
%and $U$ has maximal nilpotency class. \cite[Satz III.11.9(b)]{Huppert}, 
Thus the Chevalley commutator formula, 
see \cite[Proposition 23.11]{MalleTesterman}, shows that $G$ has
has twisted Lie rank at most $2$. We first 
consider the various types of Lie rank $1$:

For type $A_1$ the subgroup $U\cong\F_q^+$ 
has the desired shape if and only if $q=4$.
For type ${}^2 A_2$ we have $U/U'\cong\F_{q^2}^+$, where $q>2$,
hence $U$ does not have the desired shape.
%the subgroup $U$ has order $q^3$
%such that $U/U'\cong\F_{q^2}^+$ and $U'=Z(U)\cong C_q$, where $q>2$, 
%
For type ${}^2 B_2$ we have $U/U'\cong\F_{q^2}^+$, where
$q^2=2^{2f+1}>2$, hence $U$ does not have the desired shape either.
%the subgroup $U$ has order $q^4$ 
%such that $U/U'\cong\F_{q^2}^+$ and $U'=Z(U)\cong\F_{q^2}^+$,

If $G$ has twisted Lie rank $2$,
then $U/U'$ has a quotient isomorphic to 
$(U_\alpha/U'_\alpha)\times(U_\beta/U'_\beta)$,
where $\alpha,\beta$ denote the fundamental roots, and we assume
$\alpha$ to be the long one if there are two distinct root lengths.
Hence we have $U_\alpha/U'_\alpha\cong C_p\cong U_\beta/U'_\beta$,
implying that $U_\alpha\cong C_p\cong U_\beta$.
We again consider the various types:

For type $A_2$ we have $G\cong\SL_3(q)$ and 
$U_\alpha\cong\F_q^+\cong U_\beta$, hence $q=2$, in which case from
$U'=Z(U)=U_{\alpha+\beta}\cong\F_q^+$ we get $U\cong D_8$.
For types $B_2$ and $G_2$, where $B_2(q)\cong\Sp_4(q)$,
we have $U_\alpha\cong\F_q^+\cong U_\beta$, where $q>2$,
hence $U$ does not have the desired shape. 
For type ${}^2 A_3$ we have $G\cong\SU_4(q)$, 
where $U_\alpha\cong\F_q^+$ and $U_\beta\cong\F_{q^2}^+$.
For type ${}^2 A_4$ we have $G\cong\SU_5(q)$,  
where $U_\alpha\cong\F_{q^2}^+$ and $U_\beta/U'_\beta\cong\F_{q^2}^+$.
%$U_\beta$ has order $q^3$ 
%
For type ${}^3 D_4$ we have $U_\alpha\cong\F_q^+$ and
$U_\beta\cong\F_{q^3}^+$.
For type ${}^2 F_4$ we have $q=\sqrt{2}^{2f+1}$, for some $f\geq 1$, 
and $U_\alpha\cong\F_{q^2}^+$ and $U_\beta/U'_\beta\cong\F_{q^2}^+$.
%where $q^2=2^{2f+1}>2$. 
%$U_\beta$ has order $q^4$ 
\end{proof}

\begin{noth}{\bf Blocks.}\label{blocks}
Given a finite group $G$ of Lie type, 
then the $p$-blocks of $kG$ are well-understood; see for example 
\cite[Chapter 8]{Humphreys2006} or \cite[Corollary]{DipperII}:
There is a unique $p$-block of defect $0$, 
its ordinary character being the Steinberg character. 
All other $p$-blocks have maximal defect, and if $k$ contains 
$|Z(G)|$-th primitive roots of unity, there are precisely $|Z(G)|$ of them.
In particular, the principal $p$-block is the only block of positive
defect of a simple or an almost simple finite group of Lie type.

Moreover, excluding the cases explicitly mentioned in Proposition
\ref{finitetame}, it follows from the comments in
\ref{AR} below that any block algebra of $kG$ of positive
defect has wild representation type.
\end{noth}

\begin{noth}{\bf Auslander-Reiten theory.}\label{AR}
We recall the necessary facts from Auslander-Reiten theory. 
Let $G$ be any finite group, and let $B$ be a 
$p$-block algebra of $kG$, where $k$ is assumed to be algebraically closed.
Then it is well-known, see for example \cite[Theorem 4.4.4]{BensonI},
that $B$ has wild representation type if and only if the defect groups
of $B$ are neither cyclic, dihedral, semi-dihedral nor generalized quaternion.
%(where the latter three cases only apply for $p=2$). 

In this case, by \cite[Theorem 1]{Erdmann}, any connected component 
of the stable Auslander-Reiten quiver of $B$ has tree class $A_\infty$.
Hence it makes sense to discuss whether or not 
a non-projective indecomposable $B$-module $M$ lies at the end of
its connected component in the stable Auslander-Reiten quiver; 
see for example \cite[Section 2.2]{Jost}
or \cite[Introduction]{Kawata1997}. 
%for the terminology such as
%"the Auslander-Reiten component" and "lying at the end").
As the Heller operator $\Omega$
induces an automorphism of the stable Auslander-Reiten quiver of $B$,
the module $M$ lies at the end of 
its connected component in the stable Auslander-Reiten quiver,
if and only if any and hence all of its Heller translates $\Omega^i(M)$,
for $i\in\Z$, have this property.

In particular, if $S$ is a simple $B$-module, then the Auslander-Reiten
sequence ending in $\Omega^{-1}(S)$ is the standard short exact sequence
$$ \{0\}\rightarrow\Omega(S)\rightarrow
P(S)\oplus\CH(S)\rightarrow\Omega^{-1}(S)\rightarrow\{0\} ,$$
where $\CH(S):=\rad(P(S))/\soc(P(S))$ is the heart of $P(S)$.
Hence $S$ lies at the end of its connected component in the
stable Auslander-Reiten quiver, if and only if $\CH(S)$ is indecomposable.

Another sufficient condition to ensure that $S$ lies at the end of
its connected component in the stable Auslander-Reiten quiver
is given in the following Proposition \ref{KMUprop}. 
Actually, it is the key ingredient in the proof of 
Theorem \ref{KawataMichlerUno}, and its immediate 
Corollary \ref{KMUcor} is used in \ref{lost} below:
\end{noth}

\begin{Proposition}
{\bf \cite[Proposition 2.1]{KawataMichlerUno2001}.}
\label{KMUprop}
Let $G$ be a finite group, and let $U\leq G$ be a Sylow $p$-subgroup,
which is neither cyclic, dihedral, semi-dihedral nor generalized quaternion.
Let $S$ be a simple $kG$-module, such that $U$ is a vertex of $S$, and
%where the restriction $S{\downarrow}_U$ of $S$ to $U$ is a source of $S$, 
$$ \dim_k(\Hom_{kU}(k_U,S{\downarrow}_U))=1
=\dim_k(\Hom_{kU}(S{\downarrow}_U,k_U)) .$$
Then $S$ lies at the end of its connected component in the stable
Auslander-Reiten quiver of $kG$.
\end{Proposition}

\begin{proof}[\bf Proof]
By assumption, $kU$ has wild representation type.
Moreover, the second condition on $S$ implies that $S{\downarrow}_U$ is
indecomposable. %hence $S{\downarrow}_U$ is a source of $S$. 
Now the assertion follows from the argument in 
\cite[Section 3]{KawataMichlerUno2001} in conjunction with
\cite[Proposition 1.3]{KawataMichlerUno2001}.
\end{proof}

\begin{Corollary}\label{KMUcor}
%Let $U\leq G$ be a Sylow $p$-subgroup, which is neither
%cyclic, dihedral, semi-dihedral nor generalized quaternion, and
Keeping the assumptions on $G$ and $U$, let $S$ be a simple $kG$-module.

(a)
Let $\Phi(U)$ be the Frattini subgroup of $U$. Suppose that  
$S{\downarrow}_{\Phi(U)}$ is indecomposable, and that
$$ \dim_k(\Hom_{kU}(k_U,S{\downarrow}_U))=1
=\dim_k(\Hom_{kU}(S{\downarrow}_U,k_U)) .$$
Then $S$ lies at the end of its connected component in the stable
Auslander-Reiten quiver of $kG$.

(b)
Let $H$ be a finite group containing $G$ as a normal subgroup,
such that $S$ extends to a (simple) $kH$-module $T$. Moreover, assume 
that $S$ fulfills the conditions in (a). Then $T$ lies at the end of
its connected component in the stable Auslander-Reiten quiver of $kH$.
\end{Corollary}

\begin{proof}[\bf Proof]
(a)
It follows from the second condition that $S{\downarrow}_U$ is
indecomposable. Assume there is a maximal subgroup $X<U$ such that
$S{\downarrow}_U$ is relatively $X$-projective. Then
$S{\downarrow}_U$ is a direct summand of $(S{\downarrow}_X){\uparrow}^U$.
From $\Phi(U)<X$ we get that $S{\downarrow}_X$ is indecomposable, 
hence by Green's indecomposability theorem we infer that 
$(S{\downarrow}_X){\uparrow}^U$ is indecomposable.
Thus we have $S{\downarrow}_U \cong (S{\downarrow}_X){\uparrow}^U$
as $kU$-modules, a contradiction.
Hence $U$ is a vertex of $S$, and the assertion follows 
from Proposition \ref{KMUprop}.

(b)
Let $V$ be a Sylow $p$-subgroup of $H$, where we may assume
that $U=G\cap V$. Then we have $U\unlhd V$, and thus 
by \cite[Hilfssatz III.3.3(b)]{Huppert} we get $\Phi(U)\unlhd\Phi(V)$,
implying that $T{\downarrow}_{\Phi(V)}$ is indecomposable.
Moreover, since  
$\Hom_{kV}(k_V,T{\downarrow}_V)\neq\{0\}\neq\Hom_{kV}(T{\downarrow}_V,k_V)$
anyway, from the assumptions on $S$ we infer that 
$$ \dim_k(\Hom_{kV}(k_V,T{\downarrow}_V))=1
=\dim_k(\Hom_{kV}(T{\downarrow}_V,k_V)) .$$
Hence the assertion follows from (a), applied to the $kH$-module $T$.
\end{proof}

The key to prove our main theorem will be the following result:

\begin{Theorem}
{\bf Kawata--Michler--Uno \cite[Theorem]{KawataMichlerUno2001}.} 
\label{KawataMichlerUno}
Let $G$ be a quasi-simple, an almost-simple or a simple finite group 
of Lie type (in the sense of \ref{titsthm}), defined over a finite 
field of characteristic $p$, such that its Sylow $p$-subgroups are
neither cyclic, dihedral, semi-dihedral nor generalized quaternion, that is
$$ G\not\in\{\SL_2(p),\SL_2(4),\SL_3(2),\Sp_4(2)',{}^2 G_2(\sqrt{3})'\} .$$
Then any non-projective simple $kG$-module, where $k$ is algebraically closed,
lies at the end of its connected component in the stable
Auslander-Reiten quiver.
\end{Theorem}

\begin{proof}[\bf Proof]
Some care in applying in \cite[Section 3]{KawataMichlerUno2001}
has to be excercised: The strategy of proof is to ensure that
the conditions listed in Proposition \ref{KMUprop} are fulfilled. 
These in turn follow from the results in \cite{DipperI,DipperII},
including their proofs.
Now checking the assumptions made there it turns out that the 
admissible groups are precisely the quasi-simple groups
(in the sense of \ref{titsthm}). Hence this covers the quasi-simple  cases,
and it remains to consider the almost simple groups and their
derived subgroups, which is done in \ref{lost} below.
\end{proof}

%\begin{Theorem}{\bf \cite[Theorem]{KawataMichlerUno2001}.}
%\label{KawataMichlerUno}
%Assume that $p$ is an odd prime, and that
%$G:=G(q)$ is a finite group of Lie type
%defined over a finite field $\mathbb F_q$ of $q$ elements with
%a power $q$ of $p$. Further suppose that
%$B$ is a block algebra of $kG$ with full defect
%and of wild representation type. Then any simple $kG$-module $S$
%in $B$ lies at the end of its Auslander-Reiten component $\Theta (S)$.
%\end{Theorem}

\begin{noth}{\bf The lost cases.}\label{lost}
We discuss the cases excluded in Theorem \ref{Lie} and
the Kawata--Michler--Uno Theorem \ref{KawataMichlerUno},
and show that indeed the statements do {\it not} hold in these cases.
Moreover, we complete the proof of Theorem \ref{KawataMichlerUno} by
considering %explicitly constructed projective indecomposable modules for 
the relevant almost simple groups and their derived subgroups.

In order to proceed, we assume $k$ to be algebraically closed.
Given a finite group $G$,
by \cite[Theorem E]{Webb1982} the heart $\CH(k_G)$ of $P(k_G)$ is 
decomposable if and only if $p=2$ and the Sylow $p$-subgroups of $G$
are dihedral, including the Klein $4$-group.
In this case, there are three simple $kG$-modules belonging
to the principal $2$-block, and the structure of the associated projective
indecomposable modules has been determined in
\cite[Theorems 2 and 4]{Erdmann1977}. 
If the Sylow $p$-subgroups of $G$ are cyclic, then the theory of blocks 
cyclic defect applies, %\cite[Chapter VII]{Feit}, 
and the projective indecomposable $kG$-modules are described in terms of the 
Brauer trees of $kG$; see \cite{Janusz} or \cite{Kupisch}.

In contrast to this bright picture, for block algebras of wild
representation type general theory does not provide too much insight.
But the projective indecomposable modules of interest here are 
straightforwardly constructed explicitly and analysed, 
using the computer algebra systems {\sf GAP} \cite{GAP} and
{\sf C-MeatAxe} \cite{MA}, and the databases compiled in the framework
of the {\sf ModularAtlas} project \cite{ModularAtlas}.
The only exceptions are the group ${}^2 F_4(\sqrt{2})$ and its derived
subroup, the Tits group ${}^2 F_4(\sqrt{2})'$, where the projective
indecomposable modules are not easily computationally tractable due
to sheer size. Here we instead apply Corollary \ref{KMUcor} to the 
simple modules in question. Note that restriction induces a bijection
between the simple modules in the principal $p$-blocks of an almost
simple group and its derived subgroup, respectively.
We now consider the various cases:

$\circ$ 
For $G=\SL_2(p)$, where $p\geq 5$, the Sylow $p$-subgroups are cyclic
of order $p$. There are $\frac{p-1}{2}$ simple modules in each of the 
two $p$-blocks of positive defect. Their Brauer trees are straight 
lines having the exceptional vertex of multiplicity $2$ at the end, 
see \cite[Section 16.9]{Humphreys2006}.
Thus there are two simple modules, amongst them $k_G$,
such that the heart of the associated projective cover is indecomposable,
and $p-3$ simple modules not having this property. Moreover, we have
$\LL(P(k_G))=3$. 

$\circ$ 
For $G= {}^2 G_2(\sqrt{3})\cong\PSL_2(8){\colon}3$ the 
Sylow $3$-subgroups are isomorphic to the extra-special group $3_-^{1+2}$,
those of $G'= {}^2 G_2(3)'\cong\SL_2(8)\cong\PSL_2(8)$ 
are isomorphic to $C_9$. There are two simple modules $\{k_G,7\}$ 
in the principal $3$-block of $G'$.
% where we denote the non-trivial one by its degree. 
The assocoated Brauer tree is a straight 
line having the exceptional vertex of multiplicity $4$ at the end. 
Thus $\CH(k_{G'})$ is indecomposable and $\LL(P(k_{G'}))=3$, while 
$\CH(7)$ is decomposable and $\LL(P(7))=5$.
The principal $3$-block of $G$ has wild representation type,
and $\CH(S)$ turns out to be indecomposable for both its simple modules 
$S\in\{k_G,7\}$, where $\LL(P(k_G))=5$ and $\LL(P(7))=7$. 
%$\dim_k(\CH(P(7)))=94$

$\circ$ 
For $G=\SL_2(4)=\PSL_2(4)\cong\PSL_2(5)$ 
and for $G=\SL_3(2)=\PSL_3(2)\cong\PSL_2(7)$ 
the Sylow $2$-subgroups are dihedral, so that $\CH(k_G)$ is decomposable.

% assuming $k$ to contain a primitive third root of unity,
%the Sylow $2$-subgroups are isomorphic to the Klein $4$-group,
%there are three simple modules $\{1,2,2'\}$ in the principal $2$-block,
%where we denote modules by their degree, and write `$1$' for the
%trivial module $k_G$. Denoting uniserial 
%modules by giving the successive composition factors afforded
%by their unique descending composition series, we get
%$$ \CH(1)\cong [2/1/2']\oplus[2'/1/2] ,$$
%$$ \CH(2)\cong [1/2'/1] \quad\text{and}\quad \CH(2')\cong [1/2/1] .$$

%there are three simple modules $\{1,3,3'\}$ in the principal $2$-block, where
%$$ \CH(1)\cong 3\oplus 3' ,$$
%$$ \CH(3)\cong 1\oplus [3'/3/3'] \quad\text{and}\quad 
%   \CH(3')\cong 1\oplus [3/3'/3] .$$

$\circ$ 
For $G=\Sp_4(2)\cong{\mathfrak S}_6$ 
the Sylow $2$-subgroups are isomorphic to $D_8\times C_2$, 
those of $G'=\Sp_4(2)'\cong{\mathfrak A}_6\cong\PSL_2(9)$ 
are isomorphic to $D_8$. Hence $P(k_{G'})$ is decomposable.
%The block $B_0(G')$ has tame representation type, where
%$$ \CH(1)\cong [4/1/4'/1/4/1/4'] \oplus [4'/1/4/1/4'/1/4] ,$$
%$$ \CH(4)\cong [1/4'/1/4/1/4'/1] 
%\quad\text{and}\quad \CH(4')\cong [1/4/1/4'/1/4/1] .$$
The principal $2$-block of $G$ has wild representation type, 
and for all its simple modules $S\in\{k_G,4,4'\}$
the heart $\CH(S)$ turns out to be indecomposable, where $\LL(P(S))=10$.  
% $\dim_k(\CH(P(1)))=78$, $\dim_k(\CH(P(4)))=40=\dim_k(\CH(P(4')))$

$\circ$ 
For $G=G_2(2)\cong\PSU_3(3){\colon}2$ 
the Sylow $2$-subgroups have order $64$, 
%isomorphic to $C_4^2{\colon}C_2^2$,
those of $G'\cong G_2(2)'\cong \PSU_3(3)\cong\SU_3(3)$ 
have order $32$. In both cases they have nilpotency class $3$,
%isomorphic to $C_4^2{\colon}C_2$;
thus by \cite[Satz III.11.9(b)]{Huppert} they are neither 
cyclic, dihedral, semi-dihedral nor generalized quaternion. 
Hence both the principal $2$-blocks of $G'$ and $G$ 
have wild representation type.
For all the simple $kG'$-modules $S\in\{k_{G'},6,14\}$  
the heart $\CH(S)$ turns out to be indecomposable, where $\LL(P(S))=19$, 
and similarly for all the simple $kG$-modules $S\in\{k_G,6,14\}$  
the heart $\CH(S)$ turns out to be indecomposable, where $\LL(P(S))=20$.
%$\dim_k(\CH(P(1)))=446$, $\dim_k(\CH(P(4)))=500$, $\dim_k(\CH(P(16)))=292$
%$\dim_k(\CH(P(1)))=222$, $\dim_k(\CH(P(4)))=244$, $\dim_k(\CH(P(16)))=132$

$\circ$ 
For $G\cong {}^2 F_4(\sqrt{2})$
the Sylow $2$-subgroups have order $4096$,
those of ${}^2 F_4(\sqrt{2})'$ have order $2048$.
In both cases they have nilpotency class $8$, thus they are 
neither cyclic, dihedral, semi-dihedral nor generalized quaternion. 
Hence both the principal $2$-blocks of $G'$ and $G$ 
have wild representation type. Due to sheer size, in these cases
the projective indecomposable modules are not easily computationally 
tractable, so that here we revert to Corollary \ref{KMUcor} instead.
By \ref{KMUcor}(b) it suffices to check the conditions of 
\ref{KMUcor}(a) for the simple $kG'$-modules, which is straightforwardly
done explicitly. This shows that for all the simple $kG'$-modules or 
$kG$-modules $S$ the heart $\CH(S)$ is indecomposable, but does not
provide any further information on the Loewy lengths of the associated
projective indecomposable modules. Actually, with considerable computational
effort it is possible to show that for the simple $kG'$-modules 
$\{k_{G'},26,246\}$ we have $\LL(P(k_{G'}))=34$ and
$\LL(P(26))=40=\LL(P(246))$, and for the simple $kG$-modules 
$S\in\{k_G,26,246\}$ we similarly have
$\LL(P(k_{G'}))=35$ and $\LL(P(26))=41=\LL(P(246))$.

%We have checked explicitly, that $\CH(S)$ is indecomposable
%for $G'$ and $S\in\{k_{G'},26,246\}$, and 
%for $G$ and $S\in\{k_G,246\}$. For $G$ and $S\cong 26$,
%the endomorphism ring program does not seem to terminate in
%the foreseeable future.
\end{noth}

%%%%%%%%%%%%%%%%%%%%%%%%%%%%%%%%%%%%%%%%%%%%%%%%%%%%%%%%%%%%%%%%%%%%%%%%%%%%%
\section{Proofs}\label{proofs}

We are now prepared to prove the results announced in the introduction.
To do so, we will need an easy lemma, which we prove explicitly 
for convenience:

\begin{Lemma}\label{directSum}
%Let $G$ be a finite group, and let $M$ be a $kG$-module having 
%a $kG$-submodule $N$ such that both $M$ and $N$ are self-dual.
%If, furthermore, $N$ and $M/N$ do not have any (isomorphism type of)
%composition factors in common, then there is a $kG$-submodule $L$ of $M$ 
%such that $M=L\oplus N$.
Let $\iota{\colon}N\rightarrow M$ be an embedding of $A$-modules, such that
$$ \Hom_A\bigg(\, N/\rad(N),\, M/(\iota(N)+\rad(M)) \,\bigg)=\{0\} ,$$
that is the heads of $N$ and $M/\iota(N)$
do not have any composition factors in common. 
Then any epimorphism $\pi{\colon}M\rightarrow N$ is split, 
and $\iota$ is a splitting map.
\end{Lemma}

\begin{proof}[\bf Proof]
Assume that $\iota(N)+\ker(\pi)\lneq M$. Then there is a maximal $A$-submodule
$L\lneq M$ containing $\iota(N)+\ker(\pi)$. Hence the simple $A$-module $M/L$
is an epimorphic image both of $M/\iota(N)$ and of $M/\ker(\pi)\cong N$,
contradicting the assumption on composition factors. 
Hence we have $\iota(N)+\ker(\pi)=M$. Now we get
$\dim_k(\iota(N))+\dim_k(\ker(\pi))=\dim_k(N)+(\dim_k(M)-\dim_k(N))
=\dim_k(M)$, implying $\iota(N)\cap\ker(\pi)=\{0\}$,
and thus $M=\iota(N)\oplus\ker(\pi)$ as $A$-modules. 
\end{proof}

%Since $M\cong M^*$ and $N\cong N^*$ as $kG$-modules,
%there is a $kG$-epimorphism $\pi{\colon}M\rightarrow N$.
%Set $L:={\mathrm{Ker}}\,\pi$. Then $L\cap N=\{0\}$ since
%$L$ and $N$ do not have any common composition factors.
%Hence there exists a direct sum $L\oplus N$ in $M$.
%By comparing the $k$-dimensions we have to have
%$L\oplus N=M$.

Our standard application of Lemma \ref{directSum} will be the
following: Let $G$ be a finite group, and let $N\leq M$ be self-dual
$kG$-modules, such that even $N$ and $M/N$ do not have any composition 
factors in common. Then dualising the natural inclusion map
$\iota{\colon}N\rightarrow M$ gives rise to an epimorphism
$\pi{\colon}M\cong M^*\rightarrow N^*\cong N$, and hence we 
have $M\cong N\oplus(M/N)$ as $kG$-modules

%As it is realized, Lemma \ref{directSum} can be more
%generalized in terms of the socle of $N$ and the top 
%$M/[N+M{\cdot}{\mathrm{rad}}(A)]$ of $M/N$,
%instead of just their composition factors. 
%But we do not need it for our purpose, so we skip it.

\medskip
The proofs of Theorems \ref{Lie} and \ref{4_10} will both be
based on the following statement:

\begin{Proposition}\label{llprop}
Let $p$ be odd, let $G$ be a finite group such that $p \mid |G|$, and 
assume that $\LL(P(k_G))\leq 4$. Then we have $\LL(P(k_G))=3$, 
where the heart 
$$ \CH(k_G):=\rad(P(k_G))/\soc(P(k_G)) $$ 
of $P(k_G)$ 
is simple. Writing $S:=\CH(k_G)$ the following holds:

(i)
If $S\cong k_G$, then we have $p=3$, and $G$ is $3$-nilpotent
with Sylow $3$-subgroups of order $3$.

(ii)
If $S\,{\not\cong}\,k_G$, letting 
$\CH(S):=\rad(P(S))/\soc(P(S))$ be the heart of $P(S)$,
then 

$\circ$
either we have $p=3$, and $G$ is not $3$-nilpotent, has 
Sylow $3$-subgroups of order $3$, and $\CH(S)\cong k_G$,

$\circ$
or $\CH(S)$ is strictly decomposable having $k_G$ as a direct summand.
\end{Proposition}

\begin{proof}[\bf Proof]
%Set $n:=\LL(P(k_G))$. $P_1:=P(k_G)$ 
Since $\rad(\tilde{k}G)=\rad(kG)\otimes_k\tilde{k}$ for any 
separable field extension $k\subseteq\tilde{k}$, implying that
$\LL(P(k_G))=\LL(P(\tilde{k}_G))$, we can assume that 
$k$ is algebraically closed; 
see also \cite[Lemma 1.6 and its proof]{Koshitani2015}.

By Maschke's Theorem, we have $\LL(P(k_G))\geq 2$.
Suppose that $\LL(P(k_G))=2$,
then $k_G$ is the only simple $kG$-module in $B_0(G)$, 
and from the Cartan invariant $c_{11}(G)=[P(k_G)\colon k_G]=2$ we conclude
that $p=2$, a contradiction. 
%the Sylow $p$-subgroups of $G$ have order $2$, 
Similarly, it follows from \cite[Corollary]{Koshitani1996} 
that $\LL(P(k_G))\,{\not=}\,4$.

%Now we consider the self-dual module $\CH(k_G)$. 
%Then $\CH(k_G)$ is indecomposable,
%since either the Sylow $p$-subgroups of $G$ are cyclic and  
%indecomposability follows from the theory of blocks of cyclic defect,
%see for example \cite[Chapter VII]{Feit}, 
%or $B_0(G)$ has wild representation type and indecomposability
%follows from \cite[Theorem E]{Webb1982}.
%
%Suppose that $\LL(P(k_G))=4$. Then 
%since $\CH(k_G)$ has Loewy length $2$ we infer
%$$ \rad(\CH(k_G))=\soc(\CH(k_G))\cong (\CH(k_G)/\rad(\CH(k_G))^* .$$
%This implies that any self-dual composition factor of $\CH(k_G)$
%occurs with even multiplicity, contradicting \cite[Theorem]{Koshitani1996};
%see also \cite[Lemma 4.1]{Thompson}.

Hence we have $\LL(P(k_G))=3$. Thus $\CH(k_G)$ is semi-simple, 
hence is simple by \cite[Theorem E]{Webb1982}, so we write $S:=\CH(k_G)$.
If $S\cong k_G$, then $k_G$ is the only simple $kG$-module in $B_0(G)$,
hence by Brauer's Theorem \cite[Theorem V.8.3]{NagaoTsushima} we conclude
that $G$ is $p$-nilpotent, and from $c_{11}(G)=[P(k_G)\colon k_G]=3$ we 
conclude that $p=3$ and that the Sylow $p$-subgroups of $G$ have order $3$.
%Moreover, the inertial index of $B_0(G)$ is $1$.

Thus we may assume that $S\,{\not\cong}\,k_G$, and let
$\CH(S):=\rad(P(S))/\soc(P(S))$.
%Since $S$ is self-dual and occurs in the second Loewy layer of $P(k_G)$, 
%by Landrock's Lemma we infer that $k_G$ occurs in the
%second Loewy layer of $P(S)$.
Since $1=c_{1,S}(G)=[P(k_G)\colon S]=[P(S)\colon k_G]=c_{S,1}(G)$, and 
since there is a uniserial $kG$-module with composition factors $S$ and $k_G$,
%$\boxed{\begin{matrix} S\\ k_G\end{matrix}}$, 
Lemma \ref{directSum} implies that $\CH(S)=k_G\oplus\CH'$ 
for a $kG$-module $\CH'$.

Assume that $\CH'=\{0\}$. Then $\{k_G,S\}$ are the simple 
$kG$-modules in $B_0(G)$, and thus the Cartan matrix $C$ of $B_0(G)$
is of the form 
$$ C= \begin{pmatrix} 2 & 1 \\ 1 & 2 \\ \end{pmatrix} .$$
Hence we have $\det(C)=3$, thus the Sylow $p$-subgroups of $G$ 
have order $3$. Moreover, it follows from the theory of blocks of 
cyclic defect that the inertial index of $B_0(G)$ is $2$,
and thus $G$ is not $3$-nilpotent.
Therefore we are left with the case $\CH'\,{\not=}\,\{0\}$, 
that is $\CH(S)$ is decomposable of the desired form.
\end{proof}

%Actually, there are alternative ways to proceed, as soon as the
%existence of the simple $kG$-module $S$ is granted,
%see Remark \ref{malleweigel}(b) below. As was remarked earlier
%we have chosen the approach presented here, since
%it stays entirely in the realm of modular representations,
%and does not need ordinary representation theory.

\begin{proof}[\bf Proof of Theorem \ref{Lie}]
Assume to the contrary that $\LL(P(k_G))\leq 4$.
Since by Proposition \ref{finitetame} we are assuming that 
$G$ has non-cyclic Sylow $p$-subgroups, from Proposition \ref{llprop}
we infer the existence of a simple $kG$-module $S{\not\cong}k_G$, 
such that $\CH(S)$ is decomposable.
By \ref{AR} this means that $S$ does not lie at the end of
its connected component in the stable Auslander-Reiten quiver, 
contradicting the Kawata--Michler--Uno Theorem \ref{KawataMichlerUno}.
%Recall that $B_0(G)$ is of wild representation type
%since $p$ is odd and Sylow $p$-subgroups of $G$ are non-cyclic
%(see e.g. \cite[Theorem 4.4.4]{BensonI}).
%This completes the proof.
\end{proof}

\begin{proof}[\bf Proof of Corollary \ref{4_12}]
This follows immediately from Theorem \ref{Lie}.
\end{proof}

\begin{proof}[\bf Proof of Theorem \ref{4_10}]
Set $H:=O^{p'}(G)$; note that $p\mid |H|$. 
Then we have $O^{p'}(H)=H$, and $H\leq G$ being a characteristic subgroup,
from $O_{p'}(G)=\{1\}$ we infer $O_{p'}(H)=\{1\}$.

By \cite[Lemma 4.1]{KoshitaniMiyachi2001} we have $\LL(B_0(H))=4$,
where $B_0(H)$ is the principal block algebra of $kH$.
% and let $P(k_H)$ be the projective cover 
%of the trivial $kH$-module $k_H$. 
Hence we have $\LL(P(k_H))\leq 4$, and thus from Proposition \ref{llprop}
we get that the heart $S:=\CH(k_H)$ of $P(k_H)$ is simple; see also
\cite[Proposition 4.6]{KoshitaniKuelshammerSambale2014}.
Moreover, if $S\cong k_H$, then from $O_{p'}(H)=\{1\}$ we conclude
that $H\cong C_3$, hence $\LL(B_0(H))=\LL(P(k_H))=3$, a contradiction.
Hence we have $S{\not\cong}k_H$, so that the Cartan invariant 
$c_{11}(H)=2$.

We are prepared to show that $H$ is non-abelian simple:
Suppose %to the contrary 
that there exists
$N$ such that $\{1\}\,{\neq}\,N \lhd H$. % and $N\subsetneqq H$.
Then from $O_{p'}(H)=\{1\}$ and $O^{p'}(H)=H$
we conclude that both $p\mid |N|$ and $p\mid |H/N|$,
contradicting \cite[Lemma 2.5]{Koshitani2014}. 
%Thus the Alperin--Collins--Sibley Theorem \cite{ACS} 
%shows that $P(k_H)$ has a filtration with successive subquotients 
%acted on trivially by $N$, in which $P(k_{H/N})$ occurs at least twice,
%implying that $c_{11}(H)\geq 4$, a contradiction.
Moreover, if $H$ were abelian, then $H\cong C_p$, % since $p\,{|}\,|H|$,
so that $\LL(B_0(H))=\LL(kH)=p$, a contradiction.

Next we consider the Fitting subgroup $F(G)$. 
Since by assumption we have $O_{p'}(G)=\{1\}$,
it remains to consider $O_p(G)$. Since $H$ is non-abelian simple
we have $O_p(G)\cap H=\{1\}$. From this we get 
$O_p(G)\cong O_p(G)H/H \leq G/H$, where the latter is a $p'$-group,
so that $O_p(G)=\{1\}$, hence $F(G)=\{1\}$, that is
$F^*(G)=E(G)$.

Finally we consider the layer $E(G)$. 
Clearly, as $H$ is a component of $G$, we have $H\leq E(G)$.
Hence in order to claim that $H=E(G)$ we have to show that
there is no other component of $G$.
%By the definition of $E$, Letting $Q_1, \ldots, Q_n$ be the
%distinct components of $G$, for some $n\geq 1$, we can write 
%$E=Q_1 \cdots Q_n$ for distinct components $Q_1, \cdots, Q_n$ of $G$ 
%and $E$ is a central product of $Q_1 ,\cdots , Q_n$. 
Suppose to the contrary that there is a component $Q\,{\not=}\,H$ of $G$;
note that $Q\not\leq H$.
Then let $N:=\langle Q^g \mid g\in G \rangle\unlhd G$ 
be the normal closure of $Q$ in $G$. Hence we have $N\cap H=\{1\}$,
since otherwise, $H$ being simple, we had $N\cap H=H$, that is
$Q\leq N\leq H$, a contradiction.
From this we get $N\cong NH/H \leq G/H$, so that
$N$ is normal $p'$-subgroup of $G$, hence $O_{p'}(G)=\{1\}$
implies $Q\leq N=\{1\}$, a contradiction.
%Since $H\vartriangleleft G$, $Q$ normalizes $H$, so that 
%\cite[(6.9)(iv)]{Suzuki1982}
%yields that $Q$ is a component of $H$ or $[H,Q]=1$.
%If $Q$ is a component of $H$, then we have a contradiction
%since $1\,{\not=}\,Q\subn H$ and $H$ is simple.
%Since $H\unlhd G$, we have $Q^g\,{\not=}\,H$ for all $g\in G$.
%Thus, since all the subgroups $Q^g$ are components of $G$, we infer
%that $[H,Q^g]=\{1\}$ for all $g\in G$, and hence $[H,N]=\{1\}$.
%Since $H$ is simple, we have $N\cap H=H$ or $N\cap H=\{1\}$. 
%If $N\cap H=H$, that is $H\leq N$, then $[H,N]=\{1\}$ implies
%$H\leq Z(N)$, so that $H$ is abelian, a contradiction as well. 
\end{proof}

%%%%%%%%%%%%%%%%%%%%%%%%%%%%%%%%%%%%%%%%%%%%%%%%%%%%%%%%%%%%%%%%%%%%%%%%%%%%%
%\section{Concluding remarks}\label{concl}

\medskip
We end this paper with a concluding remark.

\begin{Remark}{\bf `Minimal' projective indecomposable modules.}
\label{malleweigel}
(a)
Given a finite group $G$ and a simple $kG$-module $S$,
let $c(S)$ be the integer defined by
$$ c(S):=\frac{\dim_k(P(S))}{|G|_p} ,$$ 
where $|G|_p$ is the largest power of $p$ dividing $|G|$.
Following \cite{MalleWeigel2008}, if $c(S)=1$ then 
the projective indecomposable module $P(S)$ is called minimal.

Now, given a normal subgroup $N\unlhd G$, 
Malle--Weigel \cite[Proposition 2.2]{MalleWeigel2008} show 
`supermultiplicativity' $c(k_G)\geq c(k_N)\cdot c(k_{G/N})$, 
they give sufficient conditions as to when multiplicativity holds, 
and they ask whether this is possibly always fulfilled.

However, even more generally, given any simple $kG$-module $S$ 
on which $N$ acts trivially, and denoting its deflation to $G/N$
by $\overline{S}$, it is an immediate consequence of the 
Alperin--Collins--Sibley Theorem, see \cite[Corollary 1]{ACS}, that 
%$c(S)\cdot |G|_p=\dim_k(P(S))=\dim_k(P(k_N)){\cdot}\dim_k(P(\overline{S}))
%=c(k_N){\cdot}|N|_p{\cdot}c(\overline{S}){\cdot}|G/N|_p$, and hence 
$c(S)=c(k_N)\cdot c(\overline{S})$; 
see also \cite[Lemma 2.6, Section 4]{Willems1980} and
\cite[Theorem VII.14.2]{HuppertII}.
In particular, the above question has an affirmative answer.

\medskip
(b)
We are mentioning this, as we are going to indicate two 
alternative ways to proceed in the 
final step in the proof of Theorem \ref{Lie}; they both need 
ordinary representation theory. Hence let $G$ be a simple or an almost
simple finite group of Lie type,
and recall that we know that $S:=\CH(k_G)$ is simple,
where $S{\not\cong}k_G$.
 
(i) 
Firstly, for the untwisted and twisted cases we have $|U|=q^N$, while for 
the very-twisted cases we have $|U|=(q^2)^N$, where $N$ is the number 
of positive roots in the root system associated with $\bfG$. 
We have $\dim_k(S)=c(k_G)\cdot |U|-2$, while by
\cite[Theorem 3.7]{Humphreys2006} and a similar result for the
very-twisted cases, we have $\dim_k(S)<|U|$. This implies $c(k_G)=1$, 
that is $P(k_G)$ is minimal, 
contradicting \cite[Theorem 5.8]{MalleWeigel2008};
to prove the latter ordinary character theory of
finite groups of Lie type is used. %as well.

(ii) 
Secondly, more straightforwardly, we observe that $\dim_k(S)=|U|-2$. 
Then, for the untwisted and twisted cases we may infer $N=1$,
proceeding similar to \cite[Proof of Theorem 3.7]{Humphreys2006}, 
using Weyl's character formula for ordinary Verma modules;
see also \cite[Proposition 5.1]{LassueurMalleSchulte2016}.
Thus we get $G=\SL_2(q)$, and hence $\dim_k(S)=q-2$.
Whence Steinberg's Tensor Product Theorem, see 
\cite[Theorem 2.7]{Humphreys2006}, implies $q=p$, a contradiction.
For the the very-twisted case ${}^2 G_2(\sqrt{3}^{2f+1})$ a similar
dimension estimate yields a contradiction,
see \cite[Chapter 20]{Humphreys2006}.
\end{Remark}

%%%%%%%%%%%%%%%%%%%%%%%%%%%%%%%%%%%%%%%%%%%%%%%%%%%%%%%%%%%%%%%%%%%%%%%%%%%%%
\section*{\small Acknowledgments}

{\small
A part of this work was done while the first author was visiting 
University of Jena several times in 2015--2016, 
while the second author was visiting  University of Chiba in 2015, 
and also while the authors were both visiting the
Centre Interfacultaire Bernoulli (CIB) 
at \'Ecole Polytechnique F\'ed\'erale de Lausanne (EPFL) 
during the Semester `Local representation theory and simple groups' in 2016.
The authors are grateful to these institutions for their hospitality.
The first author was partially supported by the Japan Society for 
Promotion of Science (JSPS), Grant-in-Aid for Scientific Research
(C)15K04776, 2015--2018; and also by the CIB in EPFL. %}
The second author was financially supported by the
German Science Foundation (DFG) Scientific Priority Programme
SPP-1478 `Algorithmic and Experimental
Methods in Algebra, Geometry, and Number Theory'.}

%%%%%%%%%%%%%%%%%%%%%%%%%%%%%%%%%%%%%%%%%%%%%%%%%%%%%%%%%%%%%%%%%%%%%%%%%%%%

\end{document}